\documentclass[a4paper,reqno,12pt]{amsart}
\usepackage{url}
\usepackage{amssymb,amsthm,amsmath,amsfonts,amstext,mathrsfs}
\usepackage{latexsym}
\usepackage[all]{xy}
\usepackage{graphicx}
\usepackage{enumerate}
\usepackage{color}

\newtheorem{thm}{Theorem}[section]
\newtheorem{prop}[thm]{Proposition}
\newtheorem{cor}[thm]{Corollary}
\newtheorem{lem}[thm]{Lemma}

\newcommand{\Spec}{\mathrm{Spec}}
\newtheorem{ex}[thm]{Example}
\newtheorem{defn}[thm]{Definition}
\newtheorem{rmk}[thm]{Remark}

\newtheorem{conj}[thm]{Conjecture}

\newcommand{\sA}{{\mathcal A}}
\newcommand{\sB}{{\mathcal B}}

\newcommand{\sF}{{\mathcal F}}

\newcommand{\sI}{{\mathcal I}}

\newcommand{\sO}{{\mathcal O}}

\newcommand{\sL}{\mathcal{L}}
\newcommand{\sT}{\mathcal{T}}
\newcommand{\sM}{\mathcal{M}}
\newcommand{\m}{\mathfrak{m}}
\newcommand{\md}{\mathrm{d}}

\newcommand{\px}{\dfrac{\partial}{\partial x}}

\newcommand{\sk}{\mathbf{k}}

\newcommand{\Ar}{\mathbb{A}^r_\sk}
\newcommand{\pp}{\mathbb{P}_\sk^1}
\newcommand{\pr}{\mathbb{P}^r_{\sk}}

\title{Surfaces on the Severi line in  positive characteristics}
\thanks{Yi Gu is supported by the NSFC (No. 11801391) and NSF of Jiangsu Province (No. BK20180832); Xiaotao Sun and Mingshuo Zhou are supported by the NSFC (No.11831013 and No.11501154);
Mingshuo Zhou is also supported by NSF of
Zhejiang Province (No. LQ16A010005)}
\author{Yi Gu, Xiaotao Sun and Mingshuo Zhou}
\begin{document}
\begin{abstract}
Let $X$ be a minimal surface of general type over an algebraically closed field $\sk$ of any characteristic. If the Albanese morphism $a_X:X\to \mathrm{Alb}_X$ is generically finite onto its image, we formulate a constant $c(X,L)\ge 0$ for a very ample line bundle $L$ on $\mathrm{Alb}_X$ such that $c(X,L)=0$ if and only if $\dim \mathrm{Alb}_X=2$ and $a_X: X\to \mathrm{Alb}_X$ is a double cover. A refined Severi inequality $$K^2_X\ge (4+{\rm min}\{\,c(X,L),\,\frac{1}{3}\,\})\chi(\sO_X)$$ is proved. Then we prove that $K^2_X=4\chi(\sO_X)$ if and only if the canonical model of $X$ is a flat double cover of an Abelian surface.

\end{abstract}
\maketitle

\section{Introduction}
Let $X$ be a minimal algebraic surface of general type with maximal Albanese dimension defined over an algebraically closed field. The Severi inequality asserts that
$K_X^2\ge 4\chi(\sO_X)$.
In a long time the validity of this inequality is referred to as the Severi conjecture (cf.\cite{Reid78} and \cite{Catanese83}). In \cite{Manetti03}, Manetti proved this conjecture under an extra assumption that $K_X$ is ample.  Later, Pardini \cite{Pardini05} managed to give a complete proof of Severi conjecture in characteristic zero based on her elegant covering trick. After that, Yuan and Zhang \cite{Y-Z14} generalized Pardini's result to fields of all characteristics.

\begin{thm}[Severi inequality, Pardini 05' \& Yuan-Zhang 14']
Let $X$ be a minimal surface of general type with maximal Albanese dimension, then $K_X^2\ge 4\chi(\sO_X)$.
\end{thm}

It then arises a natural question: when does the equality hold?
\begin{defn}
The surface $X$ with maximal Albanese dimension is called on the Severi line, if the equality $K_X^2=4\chi(\sO_X)$ holds.
\end{defn}
\begin{conj}[\protect{\cite[Section~5.2]{L-P12} \& \cite[Section~0]{Manetti03}}] \label{Conj: main}
A minimal surface $X$ of general type with maximal Albanese dimension is on the Severi line if and only if its canonical model $X_{\mathrm{can}}$ is a flat double cover of an Abelian surface.
\end{conj}

This conjecture was confirmed  by Barja-Pardini-Stoppino \cite{B-P-S16} and Lu-Zuo \cite{L-Z17} in characteristic zero independently.
\begin{thm}[Barja-Pardini-Stoppino 16', Lu-Zuo 17']\label{thm: B-P-S, L-Z}
Let $X$ be a minimal surface of general type with maximal Albanese dimension over an algebraically closed field of characteristic zero, then $X$ is on the Severi line if and only if the canonical model $X_{\mathrm{can}}$ is a flat double cover of an Abelian surface.
\end{thm}

The main result of this paper is a generalization of above theorem to fields of all characteristics:
\begin{thm}[Corollary~\ref{cor: double cover comes out} and Corollary~\ref{Cor: main 1}]\label{thm:main}
Let $X$ be a minimal surface of general type with maximal Albanese dimension over an arbitrary algebraically closed field, then $X$ is on the Severi line if and only if the canonical model $X_{\mathrm{can}}$ is a flat double cover of an Abelian surface.
\end{thm}

Our method is also to use Pardini's covering trick and slope inequality (the same as \cite{Pardini05} and \cite{L-Z17}), but with a lot of refinements to overcome the characteristic $p>0$ obstructions, especially when $p=2$. In characteristic $2$, firstly for the presence of purely inseparable double cover, the slope inequality in \cite{L-Z17} can fail to lead to a double cover $X\dashrightarrow Y$ factoring the Albanese morphism of $X$ (a crucial step of the method in \cite{L-Z17}).  Secondly, the double cover theory is much different in this case, and therefore the  reduction process of \cite{L-Z17} does not work well (cf. Remark~\ref{Rem: not easy to do reductions} for an explanation).

Our strategy here is to formulate and to prove a refined Severi inequality $$K^2_X\ge (4+{\rm min}\{\,c(X,L),\,\frac{1}{3}\,\})\chi(\sO_X)$$
where $c(X,L)\ge 0$ is a constant number defined as following: There are at most finite number of rational double covers $\pi_i$ ($i=1,...,s$)
$$
\xymatrix{
X\ar@{-->}[rr]^{\pi_i}\ar[rd]_{a_X} && Y_i \ar[ld]^{h_i}\\
&\mathrm{Alb}_X&
}
$$
here $Y_i$ is taking to be a minimal model. For a very ample line bundle $L$ on ${\rm Alb}_X$,  let
$c_i(X,L):=\frac{K_{Y_i}\cdot h_i^*L}{K_X\cdot a_X^*L}$ for all $i=1,...,s$. If $p=2$ and $a_X$ is separable, let $c_0(X,L):=\frac{(2K_X-R_X)\cdot L_X}{2K_X\cdot L_X}$
where $R_X:=c_1(\mathrm{det}(\Omega_{X/\mathrm{Alb}_X}))$. Then the constant $c(X,L)$ is defined to be
$$c(X,L)={\rm min}\{\,\,c_i(X,L)\,\,|\,\,0\le i\le s\,\}.$$
It is easy to see that $c(X,L)=0$ if and only if  $\dim \mathrm{Alb}_X=2$ and  the Albanese morphism $a_X:X\to {\rm Alb}_X$ is a double cover. Thus, when $K^2_X=4\chi(\sO_X)$, our inequality implies that $X$ is a double cover of an Abelian surface. After this result, in {\bf Section~\ref{Sec: 6}}, by a detailed study of flat double covers in all characteristics, we prove that $K^2_X=4\chi(\sO_X)$ holds if and only the canonical model $X_{\mathrm{can}}$ is a flat double cover of an Abelian surface.

To prove our refined Severi inequality, by Pardini's covering trick, we formulate the following commutative diagram:
$$\xymatrix{
  X_n \ar[d]_{a_n} \ar[r]^{\nu_n} & X \ar[d]^{a_X} \\
  \mathrm{Alb}_X \ar[r]^{\mu_n} & \mathrm{Alb}_X }\,\quad \,\xymatrix{
\widetilde{X}_n\ar[r]^\pi  \ar[rrd]_{\varphi_n} & X_n\ar[r]^{a_n} \ar@{-->}[dr] & \mathrm{Alb}_X\ar@{-->}[d]^{\iota_n} \\
 && \pp}$$
where $L$ is a very ample line bundle on $\mathrm{Alb}_X$, $\iota_n$ is defined by the linear pencil of dimension one in $|L|$ generated by a general member
$(B_n,B'_n)\in |L|\times |L|$ and $\pi$ is the minimal elimination of the indeterminacy. One obtains a series of fibrations $\varphi_n: \widetilde X_n\to \pp$ of genus $g_n$ with slopes
$$\lambda_{\varphi_n}=\frac{K^2_{\varphi_n}}{\chi_{\varphi_n}}=\frac{2K_X^2+6n^{-2}K_X\cdot L_X +8n^{-4}L_X^2}{2\chi(\sO_X)+n^{-2}K_X\cdot L_X+n^{-4}L_X^2}\rightarrow \frac{K_X^2}{\chi(\sO_X)}$$ (as $n\rightarrow\infty$) where $L_X=a_X^*L$. Then, by Xiao's slope inequality $$\lambda_{\varphi_n}\ge 4-\frac{4}{g_n},$$
Pardini was able to prove Severi inequality $K_X^2\ge 4\chi(\sO_X)$ in \cite{Pardini05} since $g_n\rightarrow\infty$ when $n\rightarrow\infty$. This part will be generalized to
characteristic $p>0$ in our {\bf Section~\ref{Section: Severi inequality}} ($\varphi_n$ can  inevitably have singular generic geometric fibre since the Bertini theorem is not as strong as over $\mathbb{C}$).

To prove our refined Severi inequality, we choose $(B_n,B_n')\in |L|\times |L|$ such that $\varphi_n: \widetilde X_n\to \pp$ are non-hyperelliptic fiberations of genus $g_n$, the Harder-Narasimhan filtration $$0=E_0\subset E_1\subset\cdots\subset E_{m}=(\varphi_{n})_*\omega_{\widetilde X_n/\pp}$$ defines a series of rational maps $\phi_{n,i}$
$$
\xymatrix{
\widetilde X_n\ar@{-->}[rr]^{\phi_{n,i} \ \ \ \ }\ar[rd]_{\varphi_n } && Z_{n,i}\subseteq\mathbb{P}(E_i) \ar[ld]\\
&\pp&}$$ by the canonical morphism $\varphi_n^*E_i\hookrightarrow \varphi_n^*(\varphi_{n})_*\omega_{\widetilde X_n/\pp}\to \omega_{\widetilde X_n/\pp}$.
If ${\rm deg}(\phi_{n,i})\neq 2$ for all $i$, the slope inequality of \cite{L-Z17} (see Theorem~\ref{thm: slope with gonality p} (1) for characteristic $p>0$ version) shows
$$\lambda_{\varphi_n}\ge \frac{9}{2}\frac{g_n-1}{g_n+2}.$$ Thus, when $K^2_X=4\chi(\sO_X)$, the authors of \cite{L-Z17} are able to obtains double covers $\phi_n:\widetilde X_n\dashrightarrow Z_{n,i}$ for some $i$. The key of \cite{L-Z17} is to show that these double covers led to a double cover $X\dashrightarrow Y$ factoring $a_X:X\to {\rm Alb}_X$ and then run by inductions. This argument seems not work in characteristic $2$.

Our strategy is to use another kind of slope inequality (see (2) of Theorem~\ref{thm: slope with gonality p}). Let $b_{n,i}$ be the genus of generic fiber of $Z_{n,i}\to \pp$, and $c_n:=\mathrm{min}\{\,\, \dfrac{b_{n,i}}{g_{n,i}}| \,\, \deg(\phi_{n,i})=2 \,\, \}$, we show in {\bf Section~\ref{Sec: slope inequality}} the following slope inequality
\begin{equation}\label{eq.1}
\lambda_{\varphi_n}\ge (4+{\rm min}\{\frac{1}{3},\,c_n\})\frac{g_n-1}{g_n+2}.
\end{equation}
The key technical part of {\bf Section~\ref{Sec: 5}} is devoted to show
$$\varlimsup\limits_{n\to \infty}c_n\ge  c(X,L)$$
for suitable choice of $\varphi_n$ (see Proposition \ref{prop: ratio}), which implies
$$\varlimsup\limits_{n\to \infty}\,\lambda_{\varphi_n}\ge 4+{\rm min}\{\frac{1}{3},\,c(X,L)\}$$
and thus the refined Severi inequality.

In {\bf Section~\ref{Sec: 6}}, a study of flat double cover over Abelian surface is carried out to obtain Theorem~\ref{thm:main}. Finally, in {\bf Section~7}, we give two examples of surfaces on the Severi line in characteristic $2$--one with an inseparable Albanese morphism and one with a separable one.

To make the paper more self-contained, some preliminaries are given in {\bf Section~\ref{Sec: pre}}, we suggest the readers first skip this section and then return when it is referred to.

{\bf Conventions:}
We make the following conventions in this paper:
\begin{enumerate}
\item $\sk$ is an algebraically closed field of characteristic $p>0$;

\item a surface fibration is a flat morphism $f$ from a smooth projective surface $S$ to a smooth curve $C$ such that $f_*\sO_S=\sO_C$. For such a fibration, we write  $K_f:=K_{S}-f^*K_C$ and $\chi_f:=\deg (f_*\omega_{S/C})$.

\item for a rational map $f: S\dashrightarrow T$ of $\sk$-varieties, a dominate rational map $g: S\dashrightarrow T'$ is called to be \emph{relative to $f$ or to $T$} if there is another rational map $h: T'\dashrightarrow T$ such that $f=h\circ g$. Note this $h$ is unique if exists.

%
\end{enumerate}
\section{Preliminaries}\label{Sec: pre}
\subsection{A Bertini type Theorem}
Let $X$ be a smooth projective variety over $\sk$. Let $\varphi: X\to \mathbb{P}_\sk^r$ be a non-degenerated morphism.
\begin{thm}\label{thm: Bertini}
Suppose $\varphi$ does not factor through the relative Frobenius morphism $F_{X/\sk}: X\to X^{(-1)}:=X\times_{\sk, F_\sk} \sk$ and  $\dim \varphi(X)\ge 2$, then  $\varphi^*(H)$ is a reduced and irreducible divisor for a general hyperplane $H$ in $\mathbb{P}_\sk^r$.
\end{thm}

\begin{proof}
Let $Z:=\{(x,H)\in X\times (\mathbb{P}_\sk^r)^*| \varphi(x) \in H\}$ be the incidence variety. Its first projection $p_1: Z\to X$ is a $\mathbb{P}^{r-1}$ bundle (hence $Z$ is smooth) and the second projection $p_2: Z\to (\mathbb{P}_\sk^r)^*$ is flat. The fibre $p_2^{-1}(H)$ for $H\in (\mathbb{P}_\sk^r)^*$ is by construction  $\varphi^{-1}(H)\subseteq X$. By \cite[Thm.~I.6.10(1)]{Jouanolou}, $\varphi^{-1}(H)$ is irreducible since $\dim\varphi(X)\ge 2$ for a general $H$. It then remains to show the irreducible curve $\varphi^{-1}(H)$ is smooth at some point. Note that a point $x\in \varphi^{-1}(H)$ is smooth if and only if $(x,H)\in Z$ is smooth for $p_2$. Since the smooth locus of $p_2$ on $Z$ is either empty or dominates $(\mathbb{P}_\sk^r)^*$, it suffices to prove $p_2$ is smooth at some point.

Now we  prove the general smoothness for $p_2$. Choose the standard affine subset $\Ar\subsetneq \pr$ and take $X_0:=\varphi^{-1}(\Ar)$. The morphism $\varphi: X_0\to \Ar$ is then given by $r$ regular functions $f_1,...,f_r$ as $$x\mapsto [1,f_1(x),...,f_r(x)] \in \pr.$$ Since $\varphi(X)$ is non-degenerated, $f_i$ does not vanish identically on $X_0$, we may therefore find an open dense subset $U\subseteq X_0$ where $f_1,...,f_r$ are all invertible.
Above the open affine subset $$(\mathbb{A}_\sk^r)^*=\{H_{1,t_1,...,t_r}: X_0+t_1X_1+\cdots+t_rX_r=0| t_1,...,t_r\in \sk\}\subsetneq (\mathbb{P}_\sk^r)^*,$$  the space $Z_0:=\{(x, H_{1,t_1,...,t_r})| x\in U, \varphi(x)\in H_{1,t_1,...,t_r}\}\subseteq Z$ is nothing but $\Spec U[t_1,...,t_r]/(t_1f_1+\cdots t_rf_r+1).$
Now let us calculate the relative K\"ahler differential sheaf $\Omega_{Z/(\mathbb{P}_\sk^r)^*}|_{Z_0}$. By construction, we have  $\Omega_{Z/(\mathbb{P}_\sk^r)^*}|_{Z_0}$  is isomorphic to $\Omega_{U/\sk}[t_1,...,t_r]/(t_1\mathrm{d}f_1+\cdots+ t_r\mathrm{d}f_r).$ Note that $\Omega_{U/\sk}$ is a locally free sheaf of rank $d=\dim X$ which is one more than that the relative dimension of $p_2$, thus $p_2$ is not smooth anywhere on $Z_0$ only if $t_1\mathrm{d}f_1+\cdots+ t_r\mathrm{d}f_r$ vanishes identically on $Z_0$. In other words, for any closed point $x\in U$, and any $t_1,...,t_r\in \sk$ such that $t_1f_1(x)+\cdots t_rf_r(x)=1$, we have $t_1\overline{\mathrm{d}f_1}+\cdots +t_r\overline{\mathrm{d}f_r}=0\in \Omega_{U/\sk}\otimes \sk(x).$ Assume this is the case, then for any $x\in U$, considering the point $$\xi_i=(x, (0,\cdots, f_i^{-1}, \cdots, 0))\in Z_0,$$ the vanishing of $t_1\overline{\mathrm{d}f_1}+\cdots +t_r\overline{\mathrm{d}f_r}$ now gives the vanishing of $f_i^{-1}\mathrm{d} f_i$ at $x$. By varying $x$, we have $f_i^{-1}\mathrm{d}f_i$ vanishes identically on $U$. As a consequence $\mathrm{d} f_i\equiv 0, i=1,...,r$ on $U$, the morphism $\varphi^*\Omega_{\mathbb{P}_\sk^r}\to \Omega_{X}$ also vanishes identically on $U$ and hence on $X$. It then follows that $\varphi$ factors through the relative Frobenius morphism. A contradiction to our assumption.
\end{proof}
\begin{rmk}\label{Rem: remark on Bertini}
(1). In \cite[Thm.~I.6.10]{Jouanolou}, the unramification assumption ({\it i.e.}, $\Omega_{X/\mathbb{P}_{\sk}^r}=0$) is assumed for reducibility of $\varphi^*H$.

(2). Suppose $\varphi: X\to Y\subseteq \pr$ is a finite purely inseparable morphism of smooth $d$-folds $X,Y$, then $\varphi^*H_i, i=1,...,d$ can never intersect transversely for any hyperplanes $H_1,...,H_d$.
\end{rmk}

\subsection{Inseparable double cover and foliations}\label{Subsec: foliation}
Suppose $p=2$ and $Y$ is a smooth projective surface over $\sk$. It is well known that any local derivatives $D_1, D_2$ on $Y$, $[D_1,D_2]=D_1\cdot D_2-D_2\cdot D_1$ and $D_1^2$ are again derivatives.
\begin{defn}[\protect{\cite[\S~1]{Ekedahl2}}]\label{defn: 1-foliation}
A $1$-foliation on $Y$ is a saturated subsheaf $\sF$ of the tangent sheaf $\sT_{Y/\sk}$ such that for any local derivatives $D_1,D_2$ in $\sF$, both $[D_1,D_2]$ and $D_1^2$ are also in $\sF$.
\end{defn}
\begin{thm}[\protect{\cite[Prop.~2.4]{Ekedahl}}]\label{thm: 1-1 on foliation}
There is a $1-1$ correspondence between the set of $1$-foliation $\sF$ of rank $1$ and the set of finite inseparable double cover $\pi: Y\to T$ where $T$ is normal.
\end{thm}
 This correspondence is given by
$$\{\pi: Y\to T \}\mapsto \{\sF=\sT_{Y/T}\} \text{and} \, \{\sF\subseteq \sT_{Y/\sk}\}\mapsto \{\pi: Y\to T=Y/\sF\}.$$
Now given a $1$-foliation $\sF$ of rank $1$, we obtain automatically the following exact sequence:
\begin{equation}\label{equ: exact sequence associated to a 1-foliation}
0\to \sF\to \sT_{Y/\sk}\to \sI_Z\sM\to 0,
\end{equation}
where $\sM$ is an invertible coherent sheaf and $Z$ is a scheme of finite length. The scheme $Z$ is called the singular scheme of $\sF$, it lies exactly above the singularities of $T=Y/\sF$ (cf. \cite[\S~3]{Ekedahl}). In particular, $Z=\emptyset$ if and only if $T$ is smooth.
By (\ref{equ: exact sequence associated to a 1-foliation}), we have
\begin{equation}\label{equ: for Z}
\begin{split}
\deg Z&=c_2(\sT_{Y/\sk})+c_1(\sF)\cdot (c_1(\sF)+K_Y)\\
      &=c_2(Y)+c_1(\sF)\cdot (c_1(\sF)+K_Y).
\end{split}
\end{equation}

\begin{ex}\label{exp: 1}
One typical example of a $1$-foliation of rank $1$ is obtained from a fibration. Let $f: Y\to C$ be a surface fibration, then $\sT_{Y/C}$ is a $1$-foliation of rank $1$. The finite inseparable double cover $\pi: Y\to T=Y/\sT_{Y/C}$ it associates is exactly the normalized relative Frobenius homomorphsim:
$$
\xymatrix{
Y\ar[r]^{\pi} \ar@/^2pc/[rrr]^{F_{Y/C}} \ar[rd]_f & T\ar[rr]^{\text{normalisation} \ \ \ \ }\ar[d]^g && Y\times_{C,F_{C}} C \ar[r] \ar[d] & Y\ar[d]^{f}\\
                                             & C \ar@{=}[rr]                      && C\ar[r]^{F_C}                     & C
}$$

Conversely, if $\pi: Y\to T$ is a finite inseparable double cover relative to $f: Y\to C$ with $T$ normal, then $\pi$ must be the normalized relative Frobenius as above.
%
%
\end{ex}

\section{Slope inequalities of non-hyperellitic fibrations}\label{Sec: slope inequality}
Let $f: Y\to C$ be a relatively minimal, non-hyperelliptic surface fibration of fibre genus $g\ge 2$.
Let
$$0=E_0\subseteq E_1\subseteq \cdots \subseteq E_m=f_*\omega_{Y/C}$$
be the Harder-Narasimhan filtration of $f_*\omega_{Y/C}$. Then for each $i$, there defines a natural rational map $\phi_i: Y\dashrightarrow \mathbb{P}_C(E_i)$ induced by the generically surjective morphism $f^*E_i \hookrightarrow f^*f_*\omega_{Y/C}\to \omega_{Y/C}$. Whenever $\mathrm{rank}(E_i)>1$, $Z_i:=\phi_i(Y)\subset \mathbb{P}(E_i)$ is a surface and $\phi_i$ is a generically finite morphism. In this case, denote by
\begin{itemize}
\item  $\gamma_i:=\deg(\phi_i)$ and  $b_i:=$ the fibre arithmetic genus of $Z_i\to C$.
\end{itemize} Note that $\phi_i$ is factored through by $\phi_{i+1}$ birationally, thus $\gamma_{i+1}\mid \gamma_{i}$.
%
\begin{thm}\label{thm: slope with gonality p}
Suppose $f: Y\to \pp$ is a relatively minimal, non-hyperelliptic surface fibration of fibre genus $g\ge 2$ over $\sk$.

\begin{enumerate}
\item (cf. \cite[Thm~1.2]{L-Z18} for characteristic zero) Assume that $K_f$ is nef and there is an integer $\delta$ such that either $\gamma_i=1$ or $\gamma_i>\delta$ holds for each $i$, then $$K_f^2\ge (5-\dfrac{1}{\delta})\dfrac{g-1}{g+2}\chi_f.$$

\item Assume that $K_f$ is nef and there is a constant $0<c\le \dfrac{1}{3}$ such that $b_i\ge cg$ whenever $\gamma_i=2$ then $$K_f^2\ge  (4+c)\dfrac{g-1}{g+2} \chi_f.$$
\end{enumerate}
\end{thm}
To prove this theorem, let us make some preparations first. As we are working over the base $\pp$, due to Grothendieck, we have $$f_*\omega_{Y/\pp}=\mathop{\oplus}\limits_{i=1}^m \sO(\mu_j)^{n_j}, \mu_1>\mu_2> \cdots >\mu_m .$$ By construction $E_i=\mathop{\oplus}\limits_{j=1}^{i}\sO(\mu_j)^{n_j}$ gives the Harder-Narasimhan filtration and $r_i:=\mathrm{rank}(E_i)=\sum\limits_{j=1}^in_i$.  Also we have $$\chi_f=\sum\limits_{j=1}^m \mu_i\cdot n_i=\sum\limits_{j=1}^m r_i(\mu_i-\mu_{i+1}),$$ here $\mu_{m+1}=0$.

We denote by $\sL_i:=(\mathrm{Im}(f^*E_i \hookrightarrow f^*f_*\omega_{Y/\pp} \to \omega_{Y/\pp}))^{\vee \vee}$ the  invertible coherent sheaf and $d_i:=\sL_i\cdot F,i=1,...,n$, here $F$ is a general closed fibre of $f$. As $\sL_i\otimes f^*\sO(-\mu_i)$ is generically globally generated by construction, it is nef. So Xiao's approach (cf. {\it e.g.} \cite[Prop.~2.4]{L-Z18} or \cite{S-S-Z18}) now gives
\begin{align}
\label{Xiao's approach of slope inequality} K^2_f&\geq \sum_{i=1}^m(d_i+d_{i+1})(\mu_i-\mu_{i+1}),\\
\label{Xiao's approach of slope inequality II}K^2_f&\geq (2g-2)(\mu_1+\mu_m),\\
\label{Xiao's approach of slope inequality III}K_f^2&\geq (2g-2)\mu_1
\end{align}
Here $d_{m+1}:=2g-2$.
\begin{rmk}
In positive characteristic, $\mu_m=\mu_m-\mu_{m+1}$ may be negative. So unless we know  $K_f$ is nef in a priori, \cite[Prop.~2.4]{L-Z18} works only if $i_k=m$. Namely, the last inequality  $K_f^2\geq (2g-2)\mu_1$ can fail if $K_f$ is not nef.
\end{rmk}
Next, since $f$ is non-hyperelliptic, the second multiplicative map  $$\varrho: S^2f_*\omega_{Y/C} \to f_*\omega_{Y/C}^{\otimes 2}$$ is generically surjective by Max Noether's theorem.
\begin{lem}
We have $$K_f^2=\deg(f_*\omega_{Y/\pp}^{\otimes 2})-\chi_f+l,$$
where $l:=\dim_\sk (R^1f_*\sO_X)_{\mathrm{tor}}$. In particular, we have
\begin{equation}\label{equ: K_f via second}
K^2_{f}\ge  \deg\sF-\chi_f.
\end{equation}
Here $\sF:=\mathrm{Im}(\varrho)$.
\end{lem}
\begin{proof}
Note that we have $R^1f_*\omega_{Y/\pp}^{\otimes 2}=0$, so
\begin{align*}
\chi(\sO_Y)&=\chi(\sO_{\pp})-\chi(R^1f_*\sO_Y)=\chi_f-(g-1)\chi(\sO_{\pp})- l\\
\chi(\omega_{Y/\pp}^{\otimes 2})&=\chi(f_*\omega_{Y/\pp}^{\otimes 2})=\deg f_*\omega_{Y/\pp}^{\otimes 2}+3(g-1)\chi(\sO_{\pp}).
\end{align*}
On the other hand, we have by Riemann-Roch formula
\begin{equation*}
\chi(\omega_{Y/\pp}^{\otimes 2})-\chi(\sO_Y)=K_f^2+4(g-1)\chi(\sO_{\pp}).
\end{equation*}
We are done by a simple calculation.
\end{proof}
To approach the slope inequality desired, it suffices to work out a lower bound of $\deg \sF$. Let
\begin{equation}\label{1}
0:=\mathcal{F}_0\subseteq \mathcal{F}_1\subseteq \mathcal{F}_2\subseteq \cdots \subseteq \mathcal{F}_{m-1}\subseteq \mathcal{F}_m:=\sF
\end{equation} be the filtration of $\sF$ defined as $\sF_i:=\mathrm{Im}(\varrho: S^2(E_i)\to f_*\omega^{\otimes 2}_{Y/\pp})$. Since  $\sF_i/\sF_{i+1}$ is a quotient sheaf of $S^2E_i$, its slope is at least $2\mu_i$. So we have
\begin{equation}\label{2}
\mathrm{deg}(\sF)\geq 2\sum\limits_{i=1}^m  (\mathrm{rk}(\sF_{i})-\mathrm{rk}(\sF_{i-1}))\mu_i =2\sum\limits_{i=1}^m \mathrm{rk}(\sF_i)(\mu_i-\mu_{i+1}).
\end{equation}
Now we shall turn to study the rank of each $\sF_i$. The next lemma gives a lower bound of the rank of $\sF_i$. Recall that $b_i$ by definition (at the beginning of this section) is the genus of the image $\phi_i(F)$ for a general fibre $F$.
\begin{lem}[Clifford plus theorem,  \protect{\cite[\S~III.2]{A-C-G-H}} or \protect{\cite[\S~1 ]{Harris81}}] For each $1\leq i\leq m$, we have
$$\mathrm{rk}(\mathcal{F}_i)\geq
\left\{\begin{array}{cl}
3r_i-3,& \,\,\,\, \text{if}\,\, r_i\leq b_i+1; \\
2r_i+b_i-1,& \,\,\,\, \text{if}\,\, r_i\geq b_i+2
\end{array}
\right.
$$
In particular, if $\xi_i$ is a birational morphism, then $\mathrm{rk}(\mathcal{F}_i)\geq 3r_i-3$.
\end{lem}

\begin{proof}[Proof of Theorem~\ref{thm: slope with gonality p}] We prove here only the second part, the proof of the first part is similar and one can also refer to \cite[\S~2,3]{L-Z17}.
For the second part of Theorem~\ref{thm: slope with gonality p}, we shall take $\ell'=\mathrm{min}\{i| \gamma_i=2\}$ and $\ell''=\mathrm{min}\{\ell'\le i<\ell| r_i\ge b_i+2\}$ (when $i$ ranges in $[\ell', \ell-1], b_i$ decreases and $r_i$ increases).
Then from (\ref{equ: K_f via second}), (\ref{2}) and the lemma above, we have following inequality:
\begin{equation}\label{equation K2 refined}
\begin{split}
K_f^2&\geq 2\sum\limits_{i=1}^{\ell'-1}(2r_i-1)(\mu_i-\mu_{i+1})+2\sum\limits_{i=\ell'}^{\ell''-1}(3r_i-3)(\mu_i-\mu_{i+1}) \\
        &+ 2\sum\limits_{i=\ell''}^{\ell-1}(2r_i+b_i-1)(\mu_i-\mu_{i+1})+2\sum\limits_{i=\ell}^{m}(3r_i-3)(\mu_i-\mu_{i+1})-\chi_f\\
        &\geq  \sum\limits_{i=1}^{\ell'-1}(3r_i-2)(\mu_i-\mu_{i+1})+\sum\limits_{i=\ell'}^{\ell''-1}(5r_i-6)(\mu_i-\mu_{i+1}) \\
        &+ \sum\limits_{i=\ell''}^{\ell-1}((3+2c)r_i-2)(\mu_i-\mu_{i+1})+\sum\limits_{i=\ell}^{m}(5r_i-6)(\mu_i-\mu_{i+1})
\end{split}
\end{equation}
On the other hand, by (\ref{Xiao's approach of slope inequality}) we have another inequality:
\begin{equation}\label{dffjska}
\begin{split}
K_f^2&\geq \sum\limits_{i=1}^m(d_i+d_{i+1})(\mu_i-\mu_{i+1})\\
        &\geq  \sum\limits_{i=1}^{\ell'-1}(5r_i-4)(\mu_i-\mu_{i+1})+\sum\limits_{i=\ell'}^{\ell''-1}(4r_i-4)(\mu_i-\mu_{i+1}) \\
        &+ \sum\limits_{i=\ell''}^{\ell-1}((4+2c)r_i-2)(\mu_i-\mu_{i+1})+\sum\limits_{i=\ell}^{m}(4r_i-2)(\mu_i-\mu_{i+1})-2\mu_m
\end{split}
\end{equation}
Here we note  that
\begin{itemize}
\item in case $r_1=1$, we have $(d_1+d_2)=d_2\ge (5r_1-4)=1$;

\item for $i<\ell$ and $r_i>1$, we have $d_i=\gamma_i\cdot d_i'$ where $d_i'=\deg \phi_i(F)$. The Castelnuovo's bound and Clifford's theorem then gives that:
$
d_i\ge \left\{\begin{array}{cc}
3(r_i-1),  &i<\ell';\\
2(2r_i-3), &\ell'\le i<\ell'';\\
2(r_i+b_i-1), &\ell''\le i<\ell;\\
2(r_i-1),   &\ell\le i\le m-1.
\end{array} \right.
$
\end{itemize}
Finally, $c\cdot$(\ref{equation K2 refined})$+(1-c)\cdot (\ref{dffjska})$  gives:
\begin{align*}
K_f^2&\ge (4+c)\sum\limits_{i=1}^{m}r_i(\mu_i-\mu_{i+1})+ (1-3c)\sum\limits_{i=1}^{\ell'-1}r_i(\mu_i-\mu_{i+1})\\
     &-(4-2c)\sum\limits_{i=1}^{\ell'-1}(\mu_i-\mu_{i+1})-(4+2c)\sum\limits_{i=\ell'}^{\ell''}(\mu_i-\mu_{i+1})\\
     &-(1+c)\sum\limits_{i=\ell''}^{\ell'-1}(\mu_i-\mu_{i+1})-(2+4c)\sum\limits_{i=\ell}^{m-1}(\mu_i-\mu_{i+1})-(4+2c)\mu_m\\
     &\geq (4+c)\chi_f-(4+2c)(\mu_1-\mu_m)-(4+2c)\mu_m\\
     &=(4+c)\chi_f-(4+2c)\mu_1.
\end{align*}Now the above inequality along with (\ref{Xiao's approach of slope inequality III}) gives  $K_f^2\ge (4+c)\dfrac{g-1}{g+1+c}\chi_f$, which implies our desired inequality whenever $\chi_f$ is positive or not.
\end{proof}

\section{Severi inequality in positive characteristics}\label{Section: Severi inequality}
In this section, as a preparation for studying surfaces on the Severi line, we shall recall Pardini's elegant covering trick of Severi inequality. From now on until the end of this paper, we let $X$ be a minimal algebraic surface of general type over $\sk$ and such that $X$ has maximal Albanese dimension. Namely, the Albanese morphism $a_X: X\to \mathrm{Alb}_X$ is generically finite onto its image. For our purpose, some of the original argument of Pardini is changed.

First we take a sufficiently very ample line bundle $L$ on $\mathrm{Alb}_X$ and denote by $L_X:=a_X^*L$.
For any integer $n\geq 2$ satisfying $(n,p)=1$, let $\mu_n: \mathrm{Alb}_X\rightarrow \mathrm{Alb}_X$ be the multiplication by $n$ morphism on $\mathrm{Alb}_X$ and $X_n$ be the base change as follows.
$$
\xymatrix{
  X_n \ar[d]_{a_n} \ar[r]^{\nu_n} & X \ar[d]^{a_X} \\
  \mathrm{Alb}_X \ar[r]^{\mu_n} & \mathrm{Alb}_X   }$$
Then $X_n$ is again a minimal surface of general type with maximal Albanese dimension and $$K^2_{X_n}=n^{2q}K_X^2,\,\,\,\,\,\, \chi(\sO_{X_n})=n^{2q}\chi(\sO_X).$$
Here $q:=\dim \mathrm{Alb}_X$.
As it is well known on Abelian varieties that  $$\mu^{\ast}_n(L)\sim_{\mathrm{num}}n^2L,$$ we have $L_{X_n}:=a_n^{\ast}(L)\sim_{\mathrm{num}}n^{-2}\nu_n^{\ast}(L_X)$.
Therefore $$L_{X_n}^2=n^{2q-4}L_X^2,\,\,\,\,\,\,\, K_{X_n}\cdot L_{X_n}=n^{2q-2}K_X\cdot L_X.$$
\begin{lem}[Bertini]
Let $B$ be a general member of $|L|$, then $a_n^*B$ is irreducible and reduced.
\end{lem}
\begin{proof}
See Theorem~\ref{thm: Bertini}, and note that $a_n$ does not factors through the relative Fronbenius morphism since the Albanese morphism $a_X: X\to \mathrm{Alb}_X$ is not so automatically.
\end{proof}

\begin{cor}\label{Cor: bigstar}
There is a Zariski open dense subset $U_n(X)\subseteq |L|\times |L|$ such that for any $(B_n,B_n')\in U_n(X)$, they satisfies:

($\bigstar$) both divisors $D_n:=a_n^*B_n, D_n':=a_n^*B_n'$ are irreducible, reduced and $D_n$ intersects $D_n'$ only at their common smooth locus.
\end{cor}
\begin{proof}
In fact, choose $l+1$ different general elements $\widetilde{B}_1,...,\widetilde{B}_l , \widetilde{B}'$ in $|L|$. Since $\widetilde{B}'$ is general, $\widetilde{D}_n':=a_n^*\widetilde{B}'$ will not pass through any point in the finite set  consisting of (a.) the intersection point of $\widetilde{D}_i:=a_n^*\widetilde{B}_i$ and $\widetilde{D}_j$ and (b.) the singular points of $\widetilde{D}_i,i=1,...,l$.

So unless $\widetilde{D}'$ contains at least $l$ singular points,  $\widetilde{D}'$ intersects with some $\widetilde{D}_i$ at their common smooth locus. Note that the number of singular points on $\widetilde{D}'$ is bounded by $p_a(\widetilde{D}')$ depending only on the divisor class $L$ not on the choice of $\widetilde{B}'$, so   taking any $l\ge  p_a(\widetilde{D}')+1$ at first, we shall take $B_n=\widetilde{B}_i$ for some $i$ and $B_n'=\widetilde{B}'$ to fulfil all conditions in ($\bigstar$).

Note that the existence of one single pair $(B_n,B_n')\in |L|\times |L|$ satisfying $(\bigstar)$ actually implies that any $(B_n,B_n')$ in an open dense Zariski subset $U_n(X)\subseteq |L|\times |L|$ fulfils $(\bigstar)$.
\end{proof}
Up to now, by choosing a general member $(B_n,B_n')\in |L|\times |L|$ such that $(\bigstar)$ holds, we construct the commutative diagram of (rational) morphisms in the following picture.
$$\xymatrix{
\widetilde{X}_n\ar[r]^\pi \ar[drr]_{\varphi_n}    & X_n\ar[r]^{a_n} \ar@{-->}[dr]^{\varphi_n} & \mathrm{Alb}_X \ar@{-->}[d]^{\iota_n}\\
&& \pp}$$
Here  $\iota_n$ is defined by the linear pencil of dimension one in $|L|$ generated by $B_n$ and $B_n'$, $\pi: \widetilde{X}_n\to X_n$  is the minimal elimination of the base points of $\varphi_n$   and by abuse of language, both the rational map $X_n\dashrightarrow \pp$ and  $\widetilde{X}_n\to \pp$ are called as $\varphi_n$. We write $p_1,...,p_r$ the intersections of $D_n$ and $D_n'$. Since $p_i$ is smooth on both $D_n$ and $D_n'$ by construction, we have the following lemma.
\begin{lem}
The dual graph of the exceptional divisors for $\pi: \widetilde{X}_n\to X_n$ above $p_i$ is a line (type $A_{m_i}$) as below:
$$
\xymatrix{
\bullet \ar@{-}[r]& \bullet \ar@{-}[r]& \cdots \ar@{-}[r]& \bullet \\
E_{i1} & E_{i2} &\cdots& E_{im_i}}
$$
Here $E_{i1}$ is the exceptional divisor obtained from the first blow-up, $E_{i2}$ is the second and so on.
\end{lem}

\begin{prop}\label{prop: for varphis_n}
We have
\begin{enumerate}
\item $K_{\widetilde{X}_n}=\pi^*K_{X_n}+\sum\limits_{i=1}^r (E_{i1}+2E_{i2}+\cdots+ m_iE_{im_i})$.

\item $E_{ik},k=1,...m_i-1$ are all $(-2)$-curve and $E_{im_i}$ is a $(-1)$-curve.

\item the strict transform $\widetilde{D}_n$ of $D_n$ is a fibre of $\varphi_n$. In particular, the fibration $\varphi_n$ has connected fibres.

\item $E_{im_i}$ is a section of $\varphi_n$ and $E_{ik},k=1,...,m_i-1$ are all vertical with respect to $\varphi_n$. In particular, $\varphi_n$ has no multiple fibres.

\item $K_{\varphi_n}$ is nef.
\end{enumerate}
\end{prop}
\begin{proof}
(1), (2) follows from the previous lemma.

(3) We have $\pi^*D=\widetilde{D}+\sum\limits_{i=1}^r (E_{i1}+2E_{i2}+\cdots+ m_iE_{im_i})$ and $\pi^*D'=\widetilde{D}'+\sum\limits_{i=1}^r (E_{i1}+2E_{i2}+\cdots+ m_iE_{im_i})$, so the strict transform $\widetilde{D}$ is the fibre.

(4) Note that $E_{im_i}\cdot \widetilde{D}=1$, so $E_{im_i}$ is a section. On the other hand, $E_{ik}\cdot \widetilde{D}=0$, so they are vertical.

(5) Since $K_{\widetilde{X}_n}=\pi^*K_{X_n}+\sum\limits_{i=1}^r (E_{i1}+2E_{i2}+\cdots+ m_iE_{im_i})$, we have $$K_{\varphi_n}=\pi^*K_{X_n}+\sum\limits_{i=1}^r (E_{i1}+2E_{i2}+\cdots+ m_iE_{im_i})+2F$$ with $F$ being a fibre of $\varphi_n$. It suffices to show $K_{\varphi_n} \cdot E_{ik}\ge 0$. In fact, we have
\begin{itemize}
\item $K_{\varphi_n}\cdot E_{ik}=K_{\widetilde{X}_n}\cdot E_{ik}=0, k=1,...,m_i-1$;

\item $K_{\varphi_n}\cdot E_{im_i}=K_{\widetilde{X}_n}\cdot E_{im_i}+2E_{im_i}\cdot F=-1+2=1$.
\end{itemize}
\end{proof}

As a consequence, the fibration $\varphi_n$ has the following invariants:
 \begin{equation}\label{equ: formula of varphi_n}
 g_n=\dfrac{K_{X_n}\cdot D+D^2}{2}+1=\dfrac{n^{2q-2}K_X\cdot L_X+ n^{2q-4}L_X^2}{2}+1.
\end{equation}
 $$\aligned K_{\varphi_n}^2&=K_{X_n}^2-D^2+4(K_{X_n}\cdot D+D^2)\\
           &=n^{2q}K_X^2+3n^{2q-4}L_X^2+4n^{2q-2} K_X\cdot L_X.\\ \chi_{\varphi_{n}}&=\chi(\sO_{X_n})+ \dfrac{K_{X_n}\cdot D+D^2}{2}\\&= n^{2q}\chi(\sO_X)+\dfrac{n^{2q-2}K_X\cdot L_X+ n^{2q-4}L_X^2}{2}\endaligned$$
It should be noted here
\begin{itemize}
\item $\varphi_n$ can  inevitably have singular generic geometric fibre since the Bertini theorem is not as strong as in characteristic $0$;

\item $\chi(\sO_X)$, $\chi(\sO_{X_n})$ (and hence $\chi_{\varphi_n}, n\gg 0$) are all positive integers by   \cite{Shepherd-Barron91};
\end{itemize}
So we obtain a sequence of fibred surfaces
$\varphi_n:\widetilde{X}_n\rightarrow \mathbb{P}^1$ for any $(n,p)=1$ with slopes
\begin{equation}\label{equ: formula of slope of varphin}
\begin{split}
\lambda_{\varphi_n}=\frac{K^2_{\varphi_n}}{\chi_{\varphi_n}}=\frac{2K_X^2+6n^{-2}K_X\cdot L_X +8n^{-4}H^2}{2\chi(\sO_X)+n^{-2}K_X\cdot H+n^{-4}L_X^2}\rightarrow \frac{K_X^2}{\chi(\sO_X)}.
\end{split}
\end{equation}
Note that for the fibrations $\varphi_n:\widetilde{X}_n\rightarrow \mathbb{P}^1$, all the quotients $E_i/E_{i+1}$ appearing in the Harder-Narasimhan filtration of $(\varphi_n)_{\ast}\omega_{\widetilde{X}_n/\mathbb{P}^1}$ are also strongly semistable. Thus Xiao's approach for slope inequalities works for $\mathrm{char}.(\mathbf{k})>0$  without any modification (cf. \cite{S-S-Z18}) and we have $$\lambda_{\varphi_n}\ge 4-\dfrac{4}{g_n},$$
which and (\ref{equ: formula of slope of varphin}) clearly implies $K_X^2\ge 4\chi(\sO_X)$ by letting $n\to +\infty$.

\section{Refined slopes of $X$}\label{Sec: 5}
In this section, we shall first define a constant $c(X,L)$ playing the role of $c$ in Theorem~\ref{thm: slope with gonality p}(2) for $\varphi_n$ when $n\gg 0$.
Then, we adjust the morphisms $\varphi_n$ carefully to make this constant $c(X,L)$ actually work.
Finally, by applying Theorem~\ref{thm: slope with gonality p}(2), we show that $X$ is on the Severi line only if it is a double cover of an Abelian surface.

\subsection{The constant $c(X,L)$}
First note there are only finitely many rational double covers $\pi_i: X\dashrightarrow Y_i$ relative to $a_X$ upto rational equivalence $i=1,...,s$:
$$
\xymatrix{
X\ar@{-->}[rr]^{\pi_i} \ar[rd]_{a_X}&& Y_i\ar[dl]\\
&\mathrm{Alb}_X&}
$$ where $Y_i$ is a minimal model.
In fact, the separable rational double covers relative to $a_X$ is in $1-1$ correspondence with the set of involutions of $X$ relative to $\mathrm{Alb}_X$, hence must be finite. The inseparable rational double cover relative to $a_X$ exists only if $a_X$ is inseparable itself. However, when $a_X$ is inseparable, the relative tangent sheaf $\sT_{X/\mathrm{Alb}_X}$ is a $1$-foliation (cf. Definition~\ref{defn: 1-foliation}) of rank $1$ since the relative Frobenius morphism of $X$ can never factor through the Albanese morphism. By construction the inseparable double cover $X\to X/\sT_{X/\mathrm{Alb}_X}$ (cf. Theorem~\ref{thm: 1-1 on foliation}) must factor through any other inseparable rational map $X\dashrightarrow T$ relative to $\mathrm{Alb}_X$. Therefore $X\to X/\sT_{X/\mathrm{Alb}_X}$  is the unique inseparable double cover relative to $\mathrm{Alb}_X$ upto rational equivalence.

For each $\pi_i: X\dashrightarrow Y_i, i=1,...,s$, we denote by $$c_i(X,L):=\dfrac{K_{Y_i}\cdot L_{Y_i}}{K_X\cdot L_X}, \ \  \text{here} \ \ L_{Y_i}:=h_i^*L.$$
If $p=2$ and $a_X$ is separable, there is furthermore a $c_0(X,L)$ as follows
$$c_0:=\dfrac{(2K_X-R_X)\cdot L_X}{2K_X\cdot L_X},$$
where $R_X:=c_1(\mathrm{det}(\Omega_{X/\mathrm{Alb}_X}))$.  To avoid inaccuracy, we note that $R_X$ is actually the divisor $\sum\limits_{j}\alpha_jP_j$ where
\begin{itemize}
\item $P_j$ run through all prime divisors contained in the support of $\Omega_{X/\mathrm{Alb}_X}$.

\item let $\xi_j$ be the generic point of $P_j$, then $\alpha_j:=\mathrm{legnth}(\Omega_{X/\mathrm{Alb}_X})_{\xi_j}$.
\end{itemize}

\begin{defn}\label{defn: c}
We take $c(X,L):=\mathrm{min}\{c_i(X,L)\}$ for all possible $c_i(X,L)$ defined as above and define $c(X,L)=1/2$ if $a_X$ is inseparable and no $\pi_i$ exists.
\end{defn}
\begin{prop}\label{prop: characterization of c}
The number $c(X,L)=0$ if and only if $q=2$ (recall that $q$ is the dimension of $\mathrm{Alb}_X$) and $a_X: X\to \mathrm{Alb}_X$ is a double cover.
\end{prop}
\begin{proof}
The 'if' part is clear.

For 'only if' part, first note that for $i>0$, $c_i(X,L)=0$ if and only if $K_{Y_i}\cdot L_{Y_i}=0$. As $L_{Y_i}$ is nef, big and $K_{Y_i}$ is nef, we only have $K_{Y_i}\equiv_{\mathrm{num}}0$ by Hodge index theorem. Since $Y_i$ has maximal Albanese dimension, this holds only if $Y_{i}$ is an Abelian surface. Namely we get $q=2$ and $a_X: X\to \mathrm{Alb}_X$ is a double cover.

Then for $c_0(X,L)$, when it is defined we have an exact sequence:
$$
a_X^*\Omega_{\mathrm{Alb}_X/\sk}\to \Omega_{X/\sk}\to \Omega_{X/\mathrm{Alb}_X}\to 0
$$
Denote by $\sB:=\mathrm{Im}(a_X^*\Omega_{\mathrm{Alb}_X/\sk}\to \Omega_{X/\sk})$, its double dual $(\sB)^{\vee \vee}$ is an rank $2$ locally free sheaf generically globally generated. In particular, $c_1(\sB)=c_1(\wedge^2(\sB)^{\vee \vee})$ is effective.  By the above exact sequence we have $K_X=R_X+c_1(\sB)\ge R_X$. As a consequence $2K_X-R_X\ge K_X$, and in particular we have $c_0(X,L)\ge \dfrac{1}{2}$ as $L_X$ is nef.
\end{proof}

\subsection{Refinements of $\varphi_n$}
Recall that in the construction of $\varphi_n$ in the previous section, we actually need to choose a general member $(B_n, B_n')\in |L|\times |L|$ meeting the condition $(\bigstar)$ in Corollary~\ref{Cor: bigstar}.  Namely, our $\varphi_n$ is actually defined after choosing a $\Xi\in U_n(X)$, so we shall write $\varphi_{n, \Xi}$ instead now. We have shown in the same corollary, such a choice can range in an open dense subset $U_n(X)\subset |L|\times |L|$.

\begin{lem}\label{Lem: non-hyperelliptic}
For any $n\gg 0$ and $(n,p)=1$, there is a Zariski open dense subset $V_n(X)\subseteq U_n(X)$ such that $\varphi_{n,\Xi}$ is non-hyperelliptic for all $\Xi\in V_n(X)$.
\end{lem}
\begin{proof}
Note that $\varphi_{n,\Xi}$ can never be inseparably hyperelliptic (namely, the canonical double cover is inseparable). Since otherwise the fibres of $\varphi_{n,\Xi}$ are rational, a contradiction to the maximal Albanese dimension assumption. So whenever $\varphi_{n,\Xi}$ is hyperelliptic, it gives an involution $\sigma_{n,\Xi}$ on $X_n$ relative to $\varphi_{n,\Xi}$ as $X_n$ is the minimal model.
On the other hand, the involution $\sigma_{n,\Xi}$ can clearly not be relative to $a_n$ by the maximal Albanese dimension assumption. We are done by the next lemma.
\end{proof}
\begin{lem}\label{lem:4}
Suppose for a dense subset $\Lambda_n \subseteq U_n(X)$ that each $\Xi\in \Lambda_n$ is equipped with an involution $\sigma_{n,\Xi}$ of $X_n$ relative to $\varphi_{n,\Xi}$, then there is a dense subset $\Lambda_n'\subseteq \Lambda_n$ so that $\sigma_{n,\Xi}$ is relative to $a_n$ for each $\Xi\in \Lambda_n'$.
\end{lem}
\begin{proof}
First, there are only finitely many involutions on $X_n$. Next note that for each involution $\sigma$, the set $\{\Xi\in U_n(X)|\, \sigma \ \text{is relative to} \ \varphi_{n,\Xi}\}$ is a Zariski closed subset.  And finally, $\sigma$ is relative to $a_n$ if it is relative to any $\varphi_{n,\Xi},\Xi \in U_n(X)$. Our lemma follows then.
\end{proof}
Now we choose $\Xi_n\in V_n(X)$ and therefore the associated fibration $\varphi_{n,\Xi}: \widetilde{X}_n\to \pp$ is not hyperelliptic. Applying the construction at the beginning of Section~\ref{Sec: slope inequality} to $\varphi_{n,\Xi}$, by taking the Harder-Narasimhan filtration of $(\varphi_{n,\Xi})_*\omega_{\widetilde X_n/\pp}$, we obtain a sequence of rational maps as in Section~\ref{Sec: slope inequality}: $$\phi_{n,\Xi,i}: \widetilde{X}_n \dashrightarrow Z_{n,\Xi,i}\subseteq \mathbb{P}(E_i).$$

With the help of (\ref{equ: formula of slope of varphin}) and Theorem~\ref{thm: slope with gonality p}(1) we have:
\begin{cor}\label{cor:existence of double cover}
Suppose $K_X^2 <\dfrac{9}{2}\chi(\sO_X)$, then for all $n\gg 0, (n,p)=1$ and $\Xi_n\in V_n(X)$,  some of the morphisms $\phi_{n,\Xi,i}: \widetilde{X}_n\dashrightarrow Z_{n,\Xi,i}$ are rational double cover.
\end{cor}
Fixing $n,\Xi$, there may be more than one $\phi_{n,\Xi,i}: X_n\dashrightarrow Z_{n,\Xi,i}$ of degree $2$, but they are all birationally equivalent. We then take $Z_{n,\Xi}$ to be the minimal model of all such $Z_{n,\Xi,i}$ and $\phi_{n,\Xi}: \widetilde{X}_n \dashrightarrow Z_{n,\Xi}$
$$
\xymatrix{
\widetilde{X}_n \ar[drr]_{\varphi_{n,\Xi}}\ar@{-->}[rr]^{\phi_{n,\Xi}} && Z_{n,\Xi}\ar@{-->}[d]^{\tau_{n,\Xi}} && Z'_{n,\Xi}\ar[ll]_{\vartheta_{n,\Xi}}\ar[dll]^{\tau'_{n,\Xi}} \\
&&\pp&
}$$

\begin{defn}
Taking $\vartheta_{n,\Xi}: Z'_{n,\Xi}\to Z_{n,\Xi}$ to be the minimal resolution of the indeterminancy of $\tau_{n,\Xi}$, then we denote by $g'_{n,\Xi}$ the fibre genus of $\tau_{n,\Xi}': Z'_{n,\Xi}\to\pp$.
\end{defn}
It is clear that $g'_{n,\Xi}$ is no larger than the fibre arithmetic genus of  $Z_{n,\Xi,i}\to \pp$. In the spirit of Theorem~\ref{thm: slope with gonality p}(2), to prove Theorem~{\ref{Thm: slope of X with c} below in this section, it suffices to show that $\varlimsup\limits_{n\to \infty}{\dfrac{g'_{n,\Xi}}{g_n}}\ge c(X,L)$ for a suitable choice of $\Xi_n$ to defining $\varphi_{n,\Xi_n}$ for each $n\gg 0, (n,p)=1$.

Now we start to adjust the choice of $\Xi_n$. Take $W_n(X):=V_n(X)\cap \bigcap \limits_{i=1}^s U_n(Y_i)$ ($Y_i$ is defined in the previous subsection and $U_n(Y_i)$ is defined similar to $U_n(X)$), it is a Zariski open dense subset of $|L|\times |L|$.
\begin{lem}[Lemma~3.2, \cite{L-Z17} ]\label{Lem: choice of XI}
Assume $K_X^2<\dfrac{9}{2}\chi(\sO_X)$, then for each $n\gg 0, (n,p)=1$, either \begin{enumerate}[i)]
\item there is a $\Xi\in W_n(X)$ such that $\phi_{n,\Xi}$ is separable and relative to $a_n$; or

\item there is an open dense subset $W_n(X)'\subset W_n(X)$ such that $\phi_{n,\Xi}$ is inseparable for any $\Xi\in W_n(X)'$.
\end{enumerate}
\end{lem}
\begin{proof}
Following from Lemma~\ref{lem:4}, case (i) happens if there is a dense subset of $W_n(X)$ such that $\phi_{n,\Xi}$ is separable for all $\Xi$ in this subset. Clearly, if this is not the case, we obtain (ii).
\end{proof}
We shall now choose $\Xi_n$ for each $n$ as follows.
\begin{itemize}
\item If  (i) in the above lemma happens, we choose any $\Xi_n$ making  $\phi_{n,\Xi_n}$  separable and relative to $a_n$.

\item If  (i) fails and $a_n: X\to \mathrm{Alb}_X$ is inseparable, we choose any $\Xi_n\in W_n(X)'$ making $\phi_{n,\Xi_n}$ inseparable.

\item If (i) fails and $a_n: X\to \mathrm{Alb}_X$ is inseparable, we choose  $\Xi$ as in the following lemma.
\end{itemize}
\begin{lem}\label{lemma: g'}
Suppose $a_n: X\to \mathrm{Alb}_X$ is separable and  $\phi_{n,\Xi}$ is inseparable for any $\Xi$ contained in an open dense subset $W_n(X)'\subset W_n(X)$, then there is a suitable choice of $\Xi_n$ such that  $$\dim_{\sk(t)}(\Omega_{\widetilde{X}_n/\pp,\mathrm{tor}})_\eta\le  c_1(\Omega_{X_n/\mathrm{Alb}_X})\cdot L_{X_n}.$$
for the associated fibration $\varphi_{n,\Xi_n}: \widetilde X_n \to \pp$. Here $\eta:=\Spec(\sk(t))$ is the generic point of $\pp$.
\end{lem}
This lemma is crucial to Proposition~\ref{prop: ratio} below. The inequality in this lemma is used to control the lower bound of $g_{n,\Xi}'$. Before, we prove this lemma, we first introduce some necessary notations. First denote by $X_n'\subset X_n$ the open locus where $a_n$ is unramified, and $P_1,...,P_s$ all the prime divisors contained in the complement of $X'_n$.  By Theorem~\ref{thm: Bertini} and  \cite[Thm.~I.6.10(2)]{Jouanolou}, there is an open subset $V\subset |L|$ such that for any $H\in V$, $a_n^*H$ is irreducible, reduced and moreover smooth inside $X_n'$. The choice of $\Xi_n=(B_n,B_n')$ is then as below.
\begin{enumerate}[(a)]
\item First take $W'\subseteq |L|$ to be a dense open subset contained in the projection of $W_n(X)\subseteq |L|\times |L|$ onto the first factor.

\item Then choose any $B_n\in V\cap W'$ not coinciding with any $P_j$. By construction, there is an open dense subset $M\subset |L|$ such that $B_n\times M$ is contained in $W_n(X)'$.

\item Finally choose $B_n'$ as a general member of $M$.
\end{enumerate}
\begin{proof}
Let us first note that $\dim_{\sk(t)}(\Omega_{\widetilde{X}_n/\pp,\mathrm{tor}})_\eta$ is calculated as below. First pick out all the horizontal (w.r.t. to $\varphi_{n,\Xi_n}$) divisors contained in the support of $\Omega_{\widetilde{X}_n/\pp,\mathrm{tor}}$, say $E_1,...,E_m$. Then $$\dim_{\sk(t)}(\Omega_{\widetilde{X}_n/\pp,\mathrm{tor}})_\eta=\sum\limits_{j=1}^m\mathrm{length}(\Omega_{\widetilde{X}_n/\pp, \mathrm{tor}})_{\eta_j} \cdot [E_j:\pp]$$ here $\eta_j$ is the generic point of $E_j$.

Then we shall figure out by our choice of $B_n$, $E_1,...,E_m$ is nothing but a subset of the strict transforms of the prime divisors $P_1,...,P_s$. In fact, since the exceptional divisors of $\widetilde{X}_n\to X_n$ consists of either sections of $\varphi_{n,\Xi}$ or vertical divisors by Proposition~\ref{prop: for varphis_n}, they are not contained in the horizontal component of the support of $\Omega_{\widetilde{X}_n/\pp,\mathrm{tor}}$. Next, by the special choice of $B_n$, a general member of the linear system generated by $B_n,B_n'$ lies in $V$. In particular, its pull back is smooth inside $X_n'$. Namely, for any $E_j$ its intersection with a general fibre of $\varphi_{n,\Xi}$ is contained in the inverse image of the complement of $X_0'$. As a result $E_j$ has to be one of the strict transform of $P_j$.

Finally, we calculate the length. After fixing $B_n$, the morphism $X_n \stackrel{a_n}{\to} \mathrm{Alb}_X \subseteq \mathbb{P}(H^0(L))$
restricting to $X_n\backslash D_n:=a_n^*B_n$ is a morphism to the affine space $$X_n\backslash D_n\stackrel{(f_1,...,f_l)}{\longrightarrow}\mathbb{A}^l,$$ with $l=\dim H^0(L)-1$. Then a general choice of a vector $(\lambda_1,...,\lambda_l)\in \mathbb{A}^{l}(\sk)$ gives arise to a general choice of $B_n'$ whose pull back is the zeroes of the function $\sum\limits_{i=1}^l\lambda_i f_i$. Now denote by $\xi_j$ the generic point of $P_j$.
The module $\Omega_{X_n/\sk,\xi_j}$ is a free module of $\sO_{X_n,\xi_j}$.   By above choice of $B_n'$, we have $$\Omega_{\widetilde X_n/\pp,\xi_j}=\Omega_{X_n/\sk,\xi_j}/(\sum\limits_{i=1}^l\lambda_i\md f_i).$$ In particular, its torsion length is $\mathrm{max}\{s\in \mathbb{N}| t_j^{-s}(\sum\limits_{i=1}^l\lambda_i\md f_i)\in \Omega_{X_n/\sk,\xi_j}\}$, here $t_j$ is a local parameter. It then follows from Lemma~\ref{Lemma on DVR}, that the length is not larger than the length of the torsion sheaf $$\Omega_{X_n/\sk,\xi_j}/(\md f_1,\md f_2,...,\md f_l)=\Omega_{X_n/\mathrm{Alb}_X,\xi_j}.$$ In other words, we have
\begin{align*}
\sum\limits_{j=1}^s\mathrm{length}(\Omega_{\widetilde{X}_n/\pp, \mathrm{tor}})_{\xi_j}[P_j:\pp]&\le \sum\limits_{j=1}^s\mathrm{length}(\Omega_{{X}_n/\mathrm{Alb}_X})_{\xi_j}[P_j:\pp]\\
&=c_1(\Omega_{X_n/\mathrm{Alb}_X})\cdot L_{X_n}.
\end{align*}
\end{proof}

\begin{lem}\label{Lemma on DVR}
Let $R$ be a D.V.R. containing an field $\sk$, $t$ be its uniformizer parameter. Suppose $M$ is a free $R$-module of finite rank and $u_i\in M,i=1,...,n$. For each $u\in M$ we denote by $v(u):=\mathrm{max}\{s|t^{-s}u\in M\}\in \mathbb{N} \cup +\infty$. Then,
\begin{enumerate}
\item the torsion length of $M/(u_1,...,u_n)$ is at least $\mathrm{min}\{v(u_i)|i=1,...,n\}$;

\item there is a proper linear subspace $N$ of $\sk^n$ such that for all $(\lambda_1,...,\lambda_n)\in \sk^n\backslash N$, we have $v(\sum\limits_{i=1}^n\lambda_i\cdot u_i)=\mathrm{min} \{v(u_i)|i=1,...,n\}$.
\end{enumerate}
\end{lem}
\begin{proof}
We rearrange $u_i$ such that $r=v(u_1)\le v(u_2)\le \cdots \le v(u_n)$.

(1). By construction, the canonical map $R/(t^{r}) \stackrel{\cdot t^{-r}u_1}{\to} M/(u_1,...,u_n)$ is an embedding.

(2). Dividing by $t^{r}$, we may assume that $v(u_1)=0$. As a result, we may find a basis $e_1=u_1,e_2,...,e_k$ of $M$. Then each $u_i=\sum\limits_{j=1}^k f_{ij}e_j, f_{ij}\in R,i=2,...,n$. Considering the map $$\sk^n\to R/tR: (\lambda_1,...,\lambda_n)\mapsto \lambda_1+\sum\limits_{i=2}^n\lambda_i\overline{f_{i1}}$$
This map is a non-zero $\sk$-linear map. If $(\lambda_1,...,\lambda_n)$ is not in the kernel of the this map then by construction we have
 $v(\sum\limits_{i=1}^n\lambda_i\cdot u_i)=0$.
\end{proof}

From now on, we fix a choice of $\Xi_n$ following the above rule and we drop the annoying $\Xi_n$ in subscript of the notation introduced previously for simplicity.  For example, we write $g_n'=g_{n,\Xi_n}'$ defined above.

We take $\Lambda_1:=\{n| \phi_n \ \text{is inseparable}\}$ and $\Lambda_2:=\{n| \phi_n \ \text{is  separable} \}.$

\begin{prop}\label{prop:existence of descent}
If $\Lambda_2$ contains infinitely many prime numbers, then  $\phi_n: X_n\dashrightarrow Z_n$ descends (with respective to $\nu_n:X_n\to X$) to $\pi_i: X\dashrightarrow Y_i$ for a fixed $i$ for infinitely many $n\in \Lambda_2$.
\end{prop}
In \cite{L-Z17}, they also provide a similar result \cite[Thm.~3.1]{L-Z17} by using Xiao's linear bound of the automorphism groups of surfaces of general type by $c_1^2$ (cf. \cite{Xiao94}). Such a linear bound is beyond available in positive characteristics, and we shall provide here a new argument instead.

To proceed, we need the following set up. Let $A$ be an abelian variety over $\bf{k}$, and $\mu_\ell:A\rightarrow A$ be the multiplication by $\ell$ morphism. Let $V\hookrightarrow A$ be a fixed subvariety and $V_{\ell}$ be the base change of $V$ via $\mu_\ell$:
$$
\xymatrix{
V_l\ar[rr]^{\nu_{\ell}}\ar@{^(->}[d] && V \ar@{^(->}[d] \\
A\ar[rr]^{\mu_\ell} && A
}
$$
\begin{lem}\label{Lem1}
Given any generically finite morphism $\pi: V'\to V$ (not necessarily proper) where $V'$ is a variety, assume that for infinitely many primes $\ell$ the morphism $\nu_{\ell}: V_{\ell}\to V$ factors through $\pi$, then $\pi$ is an isomorphism.
\end{lem}
\begin{proof}
We consider the field extension $K(V) \subseteq K(V')$. By our assumption, we may find some $\ell$ such that $\ell> [K(V'):K(V)]$ and the covering $\nu_\ell: V_\ell\to V$ is factored through by $\pi$ as $$\xymatrix{
V'_\ell\ar[r]^\tau \ar@/^1pc/[rr]^{\nu'_\ell} & V' \ar[r]^{\pi} & V.
}$$
Therefore $K(V')$ is $K$-linearly embedded into $K(V'_\ell)$, here $V'_\ell$ is a component of $V_\ell$. By our construction, $\mu_\ell$ (and hence $\nu_\ell$)  is a $(\mathbb{Z}/\ell\mathbb{Z})^{2g}$ Galois cover, and $V'_\ell$ is a $\Gamma$ Galois cover of $V$ for a certain subgroup $\Gamma\subseteq (\mathbb{Z}/\ell\mathbb{Z})^{2g}$. Thus we have $$[K(V'):K(V)]\cdot [K(V'_\ell):K(V')]=|\Gamma|\mid \ell^{2g}.$$ This is only possible that $[K(V'):K(V)]=1$ since $\ell\ge [K(V'):K(V)]$ by assumption.

By above argument, the morphism $\pi$ is birational. Note that $\mu_\ell$ is finite and $\tau$ is projective. Thus, $\tau$ is surjective and quasi-finite which implies it is also finite. By Chavelley's theorem, the morphism $\pi$ is also finite. On the other hand, it is clear that $\sO_{V'}\subseteq \sO_{V'_\ell}^\Gamma=\sO_V$ and we are done.
\end{proof}

\begin{proof}[Proof of Proposition~\ref{prop:existence of descent}]
We write $V:=a_X(X)$ the schemematic image of $X$ and denote by $$M:=\mathrm{Aut}_{V}(X)[2]$$ the scheme of order $2$ automorphisms of $a_X$. Since $X\to V$ is generically finite, $M$ is also generically finite over $V$.  Let $M_1, M_2,...,M_s$ be the (reduced) components of $M$ dominating $V$. By our assumption, $\nu_\ell: V_\ell\to V$ factor through $M_i$ (for some fixed $i$) for infinitely many primes  $\ell \in \Lambda_2$ represented by $\sigma_\ell$. It then follows from Lemma~\ref{Lem1} that $M_i$ is isomorphic to $V$. Namely there is a non-trivial automorphism $\sigma$ of $X$ order $2$ relative to $\mathrm{Alb}_X$ descending infinitely $\sigma_\ell$. We are done.
\end{proof}

\begin{prop}\label{Prop: inseparable descent}
If $a_X: X\to \mathrm{Alb}_X$ is inseparable, then for any $n\in \Lambda_1$, the morphism $\phi_n$ descends rationally to the unique inseparable double cover $\pi: X\to X/\sT_{X/\mathrm{Alb}_X}$.
\end{prop}
\begin{proof}
Note that when $n\in \Lambda_1$, $\phi_n$ is purely inseparable relative to the base $\pp$. So $\phi_n$ is obtained by the $1$-foliation $\sT_{X_n/\pp}$ (cf. Example~\ref{exp: 1}). However, we clearly have $\sT_{X_n/\mathrm{Alb}_X}\subseteq \sT_{X_n/\pp}$ as the morphism $\varphi_n: X_n\to \pp$ is factored through by $a_n: X_n\to \mathrm{Alb}_X$ rationally by construction. In case $a_X$ is inseparable, so is $a_n$ and hence both $1$-foliations $\sT_{X_n/\mathrm{Alb}_X}, \sT_{X_n/\pp}$ are of rank $1$. Namely, they coincides. Finally since $\mu_n$ is \'etale, $\sT_{X_n/\mathrm{Alb}_X}$ descends to $\sT_{X/\mathrm{Alb}_X}$. Therefore $\phi_n$ descends as we desired.
\end{proof}

\begin{prop}\label{prop: ratio}
We have $\varlimsup\limits_{n\to \infty}\dfrac{g_n'}{g_n}\ge  c(X,L)$.
\end{prop}
\begin{proof}
Due to Proposition~\ref{prop:existence of descent} and Proposition~\ref{Prop: inseparable descent}, there are actually two possibilities:
\begin{enumerate}
\item $\phi_n$ descends to a $\pi_i$ for infinitely many $n$;

\item $a_X$ is separable and $\phi_n$ is inseparable for infinitely many $n$.
\end{enumerate}
In this first case, the rational fibration $\tau_n: Z_n\dashrightarrow \pp$ is obtained similarly to $\varphi_n: X_n\dashrightarrow \pp$ by replacing $X_n$ by $(Y_i)_n$. By our choice of $\tau'_n$ (cf. the construction of $W_n(X)$ in this subsection), we can directly calculate that $$g_n'=\dfrac{n^{2q-2}K_{Y_i}\cdot L_{Y_i}+n^{2q-4} L_{Y_i}^2}{2}+1$$ as we do for $X_n$ (cf. formula (\ref{equ: formula of varphi_n})).
So  $g_n'/g_n\rightarrow c_i(X,L)=\dfrac{K_{Y_i}\cdot L_{Y_i}}{K_X\cdot L_X}$.

In the second case, we shall need the genus change formula. By \cite[\S~2.1, Prop.~2.2]{Gu16}, the genus $g_n'$ is given by
$$g_n'=g_n-\dfrac{1}{4}\dim_{\sk(t)} (\Omega_{\widetilde{X}_n/\pp,\mathrm{tor}})_\eta.$$
Here $\eta:=\Spec(\sk(t))$ is the generic point of the base $\pp$. By Lemma~\ref{lemma: g'}, we have
$$\dim_{\sk(t)} (\Omega_{\widetilde{X}_n/\pp,\mathrm{tor}})_\eta\le L_{X_n}\cdot c_1(\mathrm{det}(\Omega_{X_n/\mathrm{Alb}_X})).$$
Now as $\mu_n: \mathrm{Alb}_X\to \mathrm{Alb}_X$ is \'etale,  $$ c_1(\mathrm{det}(\Omega_{X_n/\mathrm{Alb}_X}))= \nu_n^* c_1(\Omega_{X/\mathrm{Alb}_X})=\nu_n^*R_X.$$
All together, the following inequality holds:
\begin{align*}
g_n'&\ge g_n-\dfrac{L_{X_n}\cdot \nu_n^*R_X}{4}\\
    &=\dfrac{n^{2q-2}K_X\cdot L_X+n^{2q-4}L_X^2}{2}-\dfrac{n^{2q-2}R_X\cdot L_X}{4}+1\\
    &=\dfrac{n^{2q-2}(2K_X-R_X)\cdot L_X+2n^{2q-4}L_X^2}{4}+1,
\end{align*}
and $\varlimsup\limits_{n\to \infty} \dfrac{g_n'}{g_n}\ge \dfrac{(2K_X-R_X)\cdot L_X}{2K_X\cdot L_X}=c_0(X,L)$.
\end{proof}

\subsection{Slopes revisited}
Immediately from Theorem~\ref{thm: slope with gonality p} and Proposition~\ref{prop: ratio}, we have the following theorem.
\begin{thm}\label{Thm: slope of X with c}
We have $K_X^2\ge [4+\mathrm{min}\{c(X,L),\dfrac{1}{3}\}]\chi(\sO_X)$.
\end{thm}
Combining with Proposition \ref{prop: characterization of c}, the following corollary is clear.
\begin{cor}\label{cor: double cover comes out}
The surface $X$ is on the Severi line only if $q=2$ and $a_X: X\to \mathrm{Alb}_X$ is a double cover.
\end{cor}

\section{Double covers of Abelian surface}\label{Sec: 6}

\subsection{Invariants of a double cover}
We start from an arbitrary morphism  $\vartheta: T\to Y$ of degree $2$ between minimal smooth surfaces over $\sk$. Taking $\vartheta_0: T_0\to Y$ to be normalisation of $Y$ in $K(T)$.
Then $\vartheta_0$ is a flat double cover. In fact, since $T_0$ is normal ${\vartheta_0}_*\sO_{T_0}$ is reflexive and hence $S_2$. It then follows from \cite{A-B57},  ${\vartheta_0}_*\sO_{T_0}$ is locally free. As a consequence, we have an exact sequence:
\begin{equation}\label{equ: exact of structures}
0\to \sO_{Y}\to {\vartheta_0}_*\sO_{T_0}\to \sL\to 0
\end{equation}
for an invertible coherent sheaf $\sL$ on $Y$. By \cite[Prop.~0.1.3]{C-D}, $T_0$ is Gorenstein and $\omega_{T_0/Y}=\vartheta_0^*\sL^{-1}$. So by Riemann-Roch formula, we have
\begin{align}
K_{T_0}^2&=2(K_Y-c_1(\sL))^2;\\
\label{equs: K^2 and chi}\chi(\sO_{T_0})&=2\chi(\sO_{Y})+\dfrac{c_1(\sL)(c_1(\sL)-K_Y)}{2}.
\end{align}
Thus
\begin{equation}\label{equ: formula of c^2-4chi}
K_{T_0}^2-4\chi(\sO_{T_0})=2(K_Y^2-4\chi(\sO_Y))+\vartheta_0^*K_Y\cdot K_{T_0/Y}.
\end{equation}

To work out the numerical invariants of $T$, we need to resolve the singularity of $T_0$. One typical way to do this is the canonical resolution (cf. \cite[\S~2]{Gu16}). Suppose $Q$ is a singularity of $T_0$ and $P=\vartheta_0(Q)\in Y$, we then blowing up $P$ by $\rho: Y'\to Y$  and take $\vartheta_0': X_0'\to Y'$ again the normalisation in $K(T)$.
$$
\xymatrix{
T_0'\ar[r]^{\rho'}\ar[d]_{\vartheta_0'} & T_0 \ar[d]^{\vartheta_0}\\
Y'\ar[r]^\rho & Y}
$$
We again have
\begin{equation}\label{equ: local use}
0\to \sO_{Y'}\to {\vartheta_0'}_*\sO_{T'_0}\to \sL'\to 0
\end{equation}
Note that $\sL'=\rho^*(\sL)(rE)$ for some $r\in \mathbb{Z}$, here $E$ is the exceptional divisor with respect to $\rho$.
\begin{lem}\label{lem: r>1}
We have $r\ge 1$.
\end{lem}
\begin{proof}
Suppose $x,y\in \m_P$ is a pair of parameter system. In this case, there is a unique $Q\in T_0$ lying above $P$. Let $z\in \m_Q$ be a function such that $\sO_{T_0,Q}=\sO_{Y,P}+z\cdot\sO_{Y,P}$. We can write out the minimal polynomial of $z$ with respective to $\sO_{Y,P}$ as $z^2+az+b=0, a,b\in \m_P.$ Then the space of relative differentials $\Omega_{T_0/\sk}\otimes\sk(Q)$ at $Q$ is given by $$\sk(Q)\mathrm{d}x\oplus \sk(Q)\mathrm{d}y\oplus \sk(Q)\mathrm{d}z/ \sk(Q)\mathrm{d} b.$$ As $Q$ is not smooth, we have $\mathrm{d}b=0$ in $\Omega_{Y,P}\otimes \sk(P)$, or equivalently $b\in \m_P^2$.

Now considering the blowing-up, in the open subset on $Y'$ where $t=\dfrac{x}{y}$ is defined ($y$ is the generator of $E$ on this open subset), we see that the minimal polynomial of $\dfrac{z}{y}$ is
$$(\dfrac{z}{y})^2-\dfrac{a}{y}\cdot \dfrac{z}{y}+\dfrac{b}{y^2}=0.$$ All coefficients are regular on this piece as $a\in \m_P$ and $b\in \m_P^2$. So $r\ge 1$.
\end{proof}

Using formula (\ref{equ: formula of c^2-4chi}), we see that
\begin{align}
\label{equ: change of chi by blowing up}
K_{T_0'}^2&=K_{T_0}^2-(r-1)^2\\
\chi(\sO_{T'_0})&=\chi(\sO_{T_0})-\dfrac{r(r-1)}{2}
\end{align}
and
\begin{equation}\label{equ: chang of K^2-4}
K_{T'_0}^2-4\chi(\sO_{T'_0})=K_{T_0}^2-4\chi(\sO_{T_0})+2(r-1).
\end{equation}

The canonical resolution process is a series of the above blowing-up and normalisation process from $T_0$ to $T_0'$. We shall write it out as the following diagram. It stops until $T_n$ is smooth over $\sk$.
$$
\xymatrix{
\dots \ar[r] & T_n\ar[d]^{\vartheta_n} \ar[r]^{\rho'_n}& \cdots \ar[r]^{\rho'_2}  \ar[d]         & T_1\ar[r]^{\rho_1' \ \ \ }\ar[d]^{\vartheta_1} & T_0 \ar[d]^{\vartheta_0}\\
 \cdots \ar[r]& Y_n\ar[r]^{\rho_n}                & \cdots \ar[r]^{\rho_2}            & Y_1\ar[r]^{\rho_1 \ \ \ } & Y_0=Y
}
$$
\begin{prop}\label{Prop:1}
The canonical resolution process stops in finitely many steps.
\end{prop}
\begin{proof}
If $p\neq 2$, this result is well known. For example one can refer to \cite[\S~2]{Gu16}.

When $p=2$, if  $\vartheta$ is inseparable, this is Proposition~\ref{prop: normalized in p=2} in below, and if $\vartheta$ is separable, this is Proposition~\ref{prop: sep}.
\end{proof}

As a consequence of Proposition~\ref{Prop:1} and formula (\ref{equ: chang of K^2-4}), we can assume $T_n$ is smooth and denote by $r_i$ such that $\sL_{i+1}=\rho_i^*\sL_i(r_iE_i)$ as before. Then we have
$$
\aligned\label{equ: 90}
K_{T}^2-4\chi(\sO_{T})\ge& K_{T_n}^2-4\chi(\sO_{T_n})\\
                      =&2(K_Y^2-4\chi(\sO_Y))+\vartheta^*_0 K_Y\cdot (K_{T_0/Y})\\
                      &+2\sum\limits_{i=1}^{n-1}(r_i-1)\\
                      \ge &2(K_Y^2-4\chi(\sO_Y))+\vartheta^*_0 K_Y\cdot (K_{T_0/Y}).
\endaligned
$$
Note that the equality
\begin{equation}
\label{equ: local for ...}K_{T}^2-4\chi(\sO_{T})=2(K_Y^2-4\chi(\sO_Y))+\vartheta^*_0 K_Y\cdot (K_{T_0/Y})
\end{equation}
holds if and only if
\begin{itemize}
\item $T_n=T$ is minimal and

\item $r_i= 1$ for $i=1,...,n-1$.
\end{itemize}
By (\ref{equ: change of chi by blowing up}), $r_i=1$ for all $i$ if and only if $\chi(\sO_{T_0})=\chi(\sO_{T_n})$, namely $T_0$ has at worst rational singularities. On the other hand, since $T_0$ is obtained by a flat double cover, its singularities are automatically a double point. So by \cite[Cor.~4.19]{Badescu01}, $T_0$ has at worst A-D-E singularities in this case. To summarize, (\ref{equ: local for ...}) holds if and only if the morphism $T\to T_{\mathrm{can}}$ to the canonical model is factored through by $T_0$. As a result, we have the following corollary.
\begin{cor}\label{Cor: main 1}
Suppose $q=2$ and $a_X: X\to \mathrm{Alb}_X$ is a double cover,  then $X$ is on the Severi line if and only if the canonical model is a flat double cover of $\mathrm{Alb}_X$.
\end{cor}
\begin{proof}
Take $T=X$, $Y=\mathrm{Alb}_X$. So (\ref{equ: local for ...}) holds if and only if $X$ is on the Severi line. Namely $T_0$ factors through the canonical morphism $X\to X_\mathrm{can}$. But since all rational curve of $X$ shall be contracted on $T_0$ it must be the canonical model itself. By construction, $T_0$ is a flat double cover of $\mathrm{Alb}_X$.
\end{proof}

\begin{rmk}\label{Rem: not easy to do reductions}
In our peculiar case, we have $Y=\mathrm{Alb}_X$ hence $K_Y=0$, hence the term $\vartheta_0^*K_Y\cdot K_{T_0/Y}=0$. In general, when $\vartheta_0$ is inseparable, $K_{T_0/Y}$ may or may not be effective and  $\vartheta_0^*K_Y\cdot K_{T_0/Y}$ can be negative. As a consequence, one can not run the reduction process in \cite{L-Z17} from formula (\ref{equ: 90}).   This reason prevents us from simulating the reduction argument of \cite{L-Z17} in characteristics $2$.
\end{rmk}

\subsection{Inseparable double covers}
In this subsubsection, we assume $\vartheta$ (hence $\vartheta_0$) is inseparable and aim to prove that the canonical resolution process stops in finitely many steps. For such $\vartheta_0: T_0\to Y$, we have another purely inseparable double cover $\pi_0: Y\to T_0^{(-1)}=T_0\times_{\sk, F_\sk} \sk$ from the commutative diagram below.
$$
\xymatrix{
 T_0 \ar[dr]_{F_{T_0/\sk}} \ar[rr]^{\vartheta_0}&& Y\ar[ld]^{\pi_0}\\
 &T_0^{(-1)} & \\
 }
$$
Note that $T_0^{(-1)}$ is isomorphic to $T_0$ as an abstract scheme with the same numerical invariants, it suffices to resolve the singularity of $T_0^{(-1)}$ and work out its own numerical invariants.

Recall that, such a purely inseparable morphism $\pi_0$ (or equivalently $\vartheta_0$) of degree $2$ is characterized by a $1$-foliation of rank $1$ on $Y$ (cf. Subsection~\ref{Subsec: foliation}). We denote by $\sF_0$ the $1$-foliation associated to $\pi_0$.

\begin{prop}[\protect{\cite[\S~3]{Ekedahl}}]\label{Prop: formula on 1-foliation}
For $\vartheta_0: T_0\to Y$, $\sL$ and $\sF_0$ as above, we have the following relation:
$$2c_1(\sL)=K_Y+c_1(\sF_0).$$
Here $\sL$ is defined as in (\ref{equ: exact of structures}).
\end{prop}
Denote by $Z_0$ the singular scheme of $\sF_0$ (cf. Subsection~\ref{Subsec: foliation}). Now reconsider $\rho: Y'\to Y$ a blowing up at a center $P\in Z_0$, denote again by $T_0'$ the normalisation of $Y'$ in $K(T)$, and define $\sF_0',\sL', Z_0'$ similar as that for $\vartheta_0: T_0\to Y$. We still denote by $r$ such that $\sL'=\rho^*\sL(rE)$ as in Lemma~\ref{lem: r>1}. Then by (\ref{equ: for Z}):
\begin{equation}
\label{equ: change of Z}\deg Z_0'=\deg Z_0-(4r^2-2r-1)<\deg Z_0.
\end{equation}

Since $Z_0$ controls the singularity of $T_0^{(-1)}$ (or equivalently $T_0$), we have given another proof of \cite[Prop.~2.6]{Hirokado99}.
\begin{prop}\label{prop: normalized in p=2}
We can resolve the singularities of $T_0$ (equivalently, that of a $1$-foliation) by a finite sequence of the normalized blowing-ups. In other words, the canonical resolution stops in finitely many steps.
\end{prop}

\subsection{Separable double covers}
Assume $\vartheta: T\to Y$ is separable, we are going to show that the process of canonical resolution stops in finitely many steps. To our purpose, we assume $\kappa(T)\ge 0$ and $\kappa(Y)\ge 0$ and hence there is an involution $\sigma: T\to T$  associated to the double cover $\vartheta: T\to Y$.
\begin{lem}\label{lem: 34}
If there is a regular model $\pi: T'\to T$ lifting then $\sigma$-action and such that $T'/\sigma$ is regular, then the canonical resolution process stops in finitely many steps.
\end{lem}
\begin{proof}
Since $Y$ is minimal, the morphism $T'/\sigma$ is obtained from $Y$ by a sequence of blowing-ups.
$$
\xymatrix{
 T'=T_n\ar[d]^{\vartheta_n} \ar[r]^{\rho'_n}& \cdots \ar[r]^{\rho'_2}  \ar[d]         & T_1\ar[r]^{\rho_1' \ \ \ }\ar[d]^{\vartheta_1} & T_0 \ar[d]^{\vartheta_0}\\
 T'/\sigma=Y_n\ar[r]^{\rho_n}                & \cdots \ar[r]^{\rho_2}            & Y_1\ar[r]^{\rho_1 \ \ \ } & Y_0=Y
}
$$
We can clearly remove the redundant intermediate blowing-up whose center is lying below a smooth point of $Y_i$. Then it becomes the canonical resolution and actually stops in at most $n$ steps.
\end{proof}
We then choose an arbitrary Lefschetz  pencil  $\iota: Y\dashrightarrow \pp$ and the associated $\tau: T\dashrightarrow \pp$. Denote by $T_1\to \pp$ the relatively minimal model of $T$. Then clearly $\sigma$ lifts to $T_1$. Denote by $Y_1=T_1/\sigma\to \pp$ the associated model. We see that $Y_1$ is regular over the generic fibre since it is normal. Now \cite[Thm.~7.3]{L-L99} applies.

\begin{thm}[\protect{\cite[Thm.~7.3]{L-L99}}]\label{thm: ll}
Let $C/\sk$ be a curve and $K:=K(C)$ be its function field. Let $f: X_K\to Y_K$ be an arbitrary Galois cover of proper regular curves over $K$ with Galois group $G$. Suppose $\mathcal{X} \to C$ is a proper regular model of $X_K$ such that the $G$-action spreads onto it. If for every closed point $x\in \mathcal{X}$, its inertia group $I_x$ has order  at most $3$, then there is a suitable choice of proper regular models $\widetilde{\mathcal{X}}\to C$ and $\widetilde{\mathcal{Y}}\to C$ such that $f$ extends to a finite Galois cover $\widetilde{f}:\widetilde{\mathcal{X}} \to  \widetilde{\mathcal{Y}}$. In other words, the $G$-action spreads to $\widetilde{\mathcal{X}}$ and its quotient $\widetilde{\mathcal{X}}/G$ is regular.
\end{thm}
In this theorem, we let $C=\pp$ and $K=\sk(t)$ be its function field and let $T_1=\mathcal{X}\to \pp, Y_K:=Y_1\times_{\pp} K$. In our case, $|G|=2$ and the inertia group assumption is satisfied automatically. It then give a model $T'=\widetilde{\mathcal{X}}$ lifting $\sigma$-action and is such that $T'/\sigma=\widetilde{\mathcal{Y}}$ is regular. So Lemma~\ref{lem: 34} gives:
\begin{prop}\label{prop: sep}
Suppose $\kappa(T)\ge 0$ and $\kappa(Y)\ge 0$, then the canonical resolution process stops in finitely many steps.
\end{prop}
\begin{rmk}
Whether or not the canonical resolution process stops in finitely many steps is a local property, so the Kodaira dimension assumption and minimality assumption is actually redundant.
\end{rmk}

\section{Examples of surfaces on the Severi line in characteristic $2$}\label{sec.7}
We give two examples of surfaces on the Severi line in characteristic $2$--one with an inseparable Albanese morphism and one with a separable one.

Denote by $$E: y^2+y=x^3+x$$  the unique supersingular elliptic curve in characteristic $2$ and we let $A:=E_1\times_\sk E_2$, here $E_1,E_2$ are two copies of $E$. We use the subscript $i=1,2$ to indicate the associated points, functions or vectors on $E_i$. We write $\Gamma_1:=\infty \times E_2$ and $\Gamma_2:=E_1\times \infty$ the two infinite divisors. Also we denote by $P=(0,0)$ and $Q=(0,1)$ the zeroes of $x$.
\subsection{Inseparable Albanese morphism}
We denote by $\partial:= \px$, the global vector field on $E$ that
$$
\left\{\begin{array}{cl}
\partial x&=1;\\
\partial y&=x^2+1
\end{array}\right.
$$
Then $\widetilde{\partial}:=x_1\partial_1+x_2\partial_2$ is a $2$-closed derivative. It generates a $1$-foliation $\sF:=K(A) \cdot \widetilde{\partial} \cap \sT_{A/\sk}$.

\begin{prop}
\begin{enumerate}
\item The $1$-foliation $\sF$ have $5$-singularities, all of which have multiplicity less than $5$;

\item The surface $X_0^{(-1)}:=A/\sF$ has $5$ A-D-E singularities and the canonical bundle is ample;

\item  Let $X$ be the minimal model of $X_0$, then $X$ is of general type and on the Severi line.
\end{enumerate}
\end{prop}
\begin{proof}
(1). By construction, it is not difficult to see that the singular scheme $Z$ of $\sF$ consists of $5$ singularities: $\infty \times \infty$, $P_1\times P_2, P_1\times Q_2, Q_1\times P_2$ and $Q_1\times Q_2$ each with multiplicity $4,1,1,1,1$ respectively.

(2). Following from (\ref{equ: change of Z}), we must have $r=1$ (otherwise the multiplicity is at least $4\cdot 2^2-2\cdot 2-1=11$) in each blowing up case. So $X_0$ has rational double points only. Note that the canonical bundle of $X_0$ is the pull back of $-\dfrac{c_1(\sF)}{2}\equiv \Gamma_1+\Gamma_2$ which is ample, therefore $X_0$ is the canonical model of a surface of general type.

(3). It follows from the formula in the previous section that $K_X^2=4\chi(\sO_X)$.
\end{proof}

\subsection{Separable Albanese morphism with wild branch divisor}
We shall construct an example of separable admissible flat double (cf. \cite[Chap.~0]{C-D})  cover $\mu_0: X_0\to A$ such that $X_0$ has at worst A-D-E singularities while the branch divisor on $A$ has singularity of arbitrarily large multiplicity. Note in other characteristics, the branch locus can have singularity of multiplicity no larger than $3$ in order the flat double cover space to have at worst A-D-E singularities.

First choose an invertible coherent sheaf $\sL$ on $A$. We define an $\sO_A$-algebra structure on $\sA:=\sO_A\oplus \sL^{-1}$ by $$\sL^{-2}\stackrel{s_2,s_1}{\longrightarrow} \sO_A \oplus \sL^{-1}=\sA.$$
Here $s_2, s_1$ are non-zero sections of $\sL^2$ and $\sL$ respectively.
In other words, let $e_i$ be a local generator of $\sL$ on an open subset $U_i$. Then
$$\sA|_{U_i}=\sO_{U_i}[z_i]/z_i^2+a_iz_i+b_i, a_i, b_i\in \sO_{U_i}.$$
Here $s_1=a_ie_i, s_2=b_ie_i^2$ on $U_i$.  And in $U_i\cap U_j$ with $e_i=\alpha_{ij}e_j$, we define $z_i=\alpha_{ij}^{-1}z_j$.

Denote by $\mu_0: X_0:=\Spec(\sA)\to A$ be the associated flat double cover and $D_i:=\mathrm{div}(s_i) \in |\sL^i|$ ($i=1,2$) the divisor associated.

\begin{prop}\label{Prop: 54}
\begin{enumerate}
\item The branch divisor of $\mu_0$ is $D_1$.

\item If $D_2$ passes through every point in the singular locus of $D_1$ smoothly, then $X_0$ has at worst A-D-E singularities.
\end{enumerate}
\end{prop}
\begin{proof}
(1). By the function mentioned above, on $U_i$ the relative K\"ahler differential $\Omega_{X_0/A}|_{U_i}$ is isomorphic to $\sA/a_i\sA$. As a result, the ramification divisor is locally defined by $a_i$. In other words, the branch divisor is defined by $a_i$ on $U_i$.

(2). Let $P\in D_1$. Suppose $D_1$ is smooth at $P$, then $X_0$ has at worst A-D-E singularity above $P$ by \cite[Remark~0.2.2]{C-D}. Now suppose $P$ is singular on $D_1$, then by our assumption $D_2$ pass through $P$ and is smooth at $P$. Namely in the local function $z_i^2+a_iz_i+b_i$ defining $X_0$ near $P$ we have $a_i\in\m_P$ and $b_i\in \m_p\backslash \m_p^2$. This clearly implies $X_0$ is regular above $P$.
\end{proof}
Conversely, given any effective divisors $D_1, D_2$ such that $D_2\in |2D_1|$, then one can construct an admissible example as above.

As a result, the singularity of the branch divisor $D_1$ does not eventually leads to the singularity of $X_0$. Along this way, we can construct examples where the branch divisor is very singular at some points but $X_0$ has at worst A-D-E singularities.
\begin{ex}
We take $D_1$ to be the divisor defined by equation: $$x_1^{2n}+x_2^{2n+1}=0.$$ Then $D_1\in|2n\Gamma_1+(2n+1)\Gamma_2|$ and the singularity of $D_1$ is again the five points $\infty \times \infty$, $P_1\times P_2, P_1\times Q_2, Q_1\times P_2$ and $Q_1\times Q_2$. Now we can construct $D_2\in |2D_1|=|4n\Gamma_1+(4n+2)\Gamma_2|$ passing through the five points smoothly. We take $D_2=E_1\times P_2+E_1\times Q_2+\Gamma_2+D_2'$ where $D_2'\in |4n\Gamma_1+(4n-1)\Gamma_2|$ is a general member which does not pass through the five points above. Here note that the linear system $|4n\Gamma_1+(4n-1)\Gamma_2|$ is very ample for all $n\in \mathbb{N}_+$.

We denote by $X_0$ the separable flat double cover of $A$ defined by $D_1, D_2$. Then $X_0$ is the canonical model of a minimal surface of general type $X$ on the Severi line by Corollary~\ref{Cor: main 1} and Proposition~\ref{Prop: 54}. In this case the branch divisor of $X_0\to A$ is wild (of multiplicity $2n$) at the four points  $P_1\times P_2, P_1\times Q_2, Q_1\times P_2$ and $Q_1\times Q_2$.
\end{ex}

\section*{\bf Acknowledgement}
The first author would like to thank D. Lorenzini for telling him Theorem~\ref{thm: ll} and T. Zhang for helpful communications.

\noindent\address{Yi Gu: School of Mathematical Sciences, Soochow University,  Suzhou
215006, P. R. of China.}\\
{\em Email}:  sudaguyi2017@suda.edu.cn

\vspace{0.2cm}

\noindent\address{Xiaotao Sun: Center of Applied Mathematics, School of Mathematics, Tianjin University, Tianjin 300072, P. R. of China}\\
{\em Email}: xiaotaosun@tju.edu.cn

\vspace{0.2cm}

\noindent\address{Mingshuo Zhou: Center of Applied Mathematics, School of Mathematics, Tianjin University,  Tianjin 300072, P. R. of China}\\
{\em Email}: zhoumingshuo@amss.ac.cn
\end{document}